\documentclass[11pt]{amsart}
\usepackage{graphicx}
\usepackage{pgf,tikz,pgfplots}
\usepackage{mathrsfs}
\usepackage{mathpple}
\usepackage{ bbold }
\usepackage{tikz-cd}
\usepackage{amsmath}
\usepackage{tikz}
\usepackage{mathdots}

\usepackage{cancel}
\usepackage{color}
\usepackage{siunitx}
\usepackage{array}
\usepackage{multirow}
\usepackage{amssymb}
\usepackage{hyperref}

\usepackage{latexsym,bm,amsmath,amssymb}

\usetikzlibrary{arrows}
\newcommand{\C}{\mathbb{C}}

\newcommand{\Z}{\mathbb{Z}}

\newcommand{\CE}{Chevalley-Eilenberg }

\newcommand\sbullet[1][.5]{\mathbin{\vcenter{\hbox{\scalebox{#1}{$\bullet$}}}}}

\DeclareMathOperator{\Tr}{Tr}

\numberwithin{equation}{section}

\newtheorem{defn}{Definition}[section]
\newtheorem{theorem}{Theorem}[section]
\newtheorem{cor}{Corollary}[section]
\newtheorem{lem}{Lemma}[section]
\newtheorem{prop}{Proposition}[section]
\newtheorem{eg}{Example}[section]

\newtheorem{remark}{Remark}[section]

\DeclareFontFamily{U}{eus}{\skewchar\font'60}
\DeclareFontShape{U}{eus}{m}{n}{%
	<-6>eusm5%
	<6-8>eusm7%
	<8->eusm10%
}{}

\DeclareMathAlphabet\EuScript{U}{eus}{m}{n}

\begin{document}
	
	\title{Loday-Quillen-Tsygan theorem on Quivers}%
	\author{Keyou Zeng}

	\address{
		K. Zeng:
		Perimeter Institute for Theoretical Physics, Waterloo, Canada
	}
	
	\email{kzeng@perimeterinstitute.ca}

	\thanks{}%

	\begin{abstract}
The well-known Loday-Quillen-Tsygan theorem calculates the Lie algebra homology of the infinite general linear Lie algebra $\mathfrak{gl}(A)$ over an unital associative algebra $A$. We generalize the Loday-Quillen-Tsygan theorem to an infinite Lie algebra associated with a (framed) quiver, where we assign to each vertex $v$ an infinite general linear Lie algebra $\mathfrak{gl}(A_v)$, to each edge $e$ an infinite matrix module and to each framed vertex a (anti)-fundamental representation. Given this data, each loop or path ending on framed vertices of the quiver defined a stratified factorization algebra over $S^1$ or $[0,1]$ respectively. We show that the corresponding Lie algebra homology can be expressed as summing the factorization homology over all loops and framed paths of the quiver.

	\end{abstract}

		\maketitle

\section{Introduction}
The celebrated Loday-Quillen-Tsygan (LQT) theorem \cite{Loday1984}, \cite{Tsygan1983} states that for any associative algebra $A$ in characteristic zero, the Lie algebra homology of the infinite general linear Lie algebra over $A$ is canonically isomorphic, as a Hopf algebra, to the exterior power of the cyclic homology of $A$:
$$
\lim_{\overset{\longrightarrow}{N}} H_{\sbullet}(\mathfrak{gl}_N(A)) \cong \mathrm{Sym}(HC_{\sbullet}(A)[1]).
$$
This theorem admit various generalization to many different setups. One can replace the Lie algebra $\mathfrak{gl}_N(A)$ by the Lie algebra of symplectic matrices $\mathfrak{sp}_N(A)$ or orthogonal matrices $\mathfrak{o}_N(A)$. Then the corresponding Lie algebra homology is computed by the dihedral homology \cite{Procesi1988}. There is also a generalization of the LQT theorem where one considers the Lie algebra homology of $\mathfrak{gl}_N(A)$ with coefficient in the adjoint matrices representation \cite{Goodwillie1985}.

In this paper, we consider an extension of the LQT theorem to an infinite Lie algebra associated with a framed quiver. Recall that a quiver $Q$ is a directed graph that consist of two sets $(Q_0,Q_1)$ and two functions
$$
s,t : Q_1 \rightrightarrows Q_0.
$$
The set $Q_0$ is the set of vertices and $Q_1$ is the set of oriented edges. An edge $e\in Q_0$ starts at the vertex $s(e)$ and ends on the vertex $t(e)$.

A framed quiver $Q^{\text{fr}}$ consists of the data of a quiver $Q = (Q_0,Q_1)$ together with two subsets $W^+,W^-$ of $Q_0$ as the framed vertices. Denote the inclusion map $j:W^+\sqcup W^- \to Q_0$. We can think of the framed quiver $Q^{\text{fr}}$ as adding to $Q$ extra vertices $W^+,W^-$ and edges $w^+ \to j(w^+)$ for all $w^+ \in W^+$, and $ j(w^-) \to w^-$ for all $w^- \in W^-$.

For a framed quiver $Q^{\text{fr}}$, we assign to it the following set of data 
$$
	\left\{
	\begin{tabular}{c}
		To each vertex $v \in Q_0$ we assign a unital $k$ algebra $A_v$\\
		To each edge $e \in Q_1$ we assign a $A_{s(e)} - A_{t(e)}$ bimodule $M_e$\\
		To each vertex $w^+ \in W^+$ we assign a right $A_{j(w^+)}$ module $M_{w^+}$\\
		To each vertex $w^- \in W^-$ we assign a left $A_{j(w^-)}$ module $M_{w^-}$
	\end{tabular}
	\right\}.
$$
Given this data, we have the following Lie algebra
$$
\bigoplus_{v\in Q_0} \mathfrak{gl}_N(A_v),
$$
and its module
\begin{equation}\label{Lie_module_fr}
	\bigoplus_{e\in Q_1}\mathrm{Mat}_N(M_e) \bigoplus_{w^{-}\in W^-}\C^N(M_{w^-})\bigoplus_{w^{+}\in W^+}(\C^N)^*(M_{w^+}).
\end{equation}
In this paper, we consider the Lie algebra homology of $\bigoplus_{v\in Q_0} \mathfrak{gl}_N(A_v)$ with coefficient in the symmetric power of the module \ref{Lie_module_fr} in the $N \to \infty$ limit
\begin{equation}
	\label{main_Lie_ho}
	\lim_{\overset{\longrightarrow}{N}}H_{\sbullet}(\bigoplus_{v\in Q_0} \mathfrak{gl}_N(A_v),\mathrm{Sym}(\bigoplus_{e\in Q_1}\mathrm{Mat}_N(M_e) \bigoplus_{w^{+}\in W^+}\C^N(M_{w^+})\bigoplus_{w^{-}\in W^-}(\C^N)^*(M_{w^-}))).
\end{equation}
Equivalently, this is isomorphic to the Lie algebra homology of the following super Lie algebra
$$
\left( \bigoplus_{v\in Q_0} \mathfrak{gl}_N(A_v)\right)  \ltimes \left( \bigoplus_{e\in Q_1}\mathrm{Mat}_N(M_e) \bigoplus_{w^{+}\in W^+}\C^N(M_{w^+})\bigoplus_{w^{-}\in W^-}(\C^N)^*(M_{w^-})\right)[-1].
$$

The main theorem we prove in this paper is the following
\begin{theorem}[Theorem \ref{thm_main1}]
	The Lie algebra homology \ref{main_Lie_ho} is isomorphic, as a Hopf algebra, to the following 
\begin{equation}\label{main_result}
	\begin{split}
		\mathrm{Sym} & \left(  \bigoplus_{v \in Q_0}HC_{\sbullet}(A_v)[1] \right. \\
		& \bigoplus_{\ell = (e_1,\dots,e_n) \in \mathrm{Cyc}(Q)}  H_{\sbullet}\left( M_{e_1}\otimes_{A_{t(e_1)}}^{\mathbb{L}}M_{e_2}\otimes_{A_{t(e_2)}}^{\mathbb{L}}\dots \otimes_{A_{t(e_{n-1})}}^{\mathbb{L}}M_{e_n}\otimes_{A_{t(e_n)}}^{\mathbb{L}}\right) _{\Z_{\deg \ell}}\\
		&\left. \bigoplus_{\rho = (w^+,e_1,\dots,w^-) \in \mathrm{fPath}(Q)} H_{\sbullet}\left( M_{w^+}\otimes^{\mathbb{L}}_{A_{j(w^+)}}M_{e_{1}}\otimes_{A_{t(e_1)}}^{\mathbb{L}} \dots M_{e_n}\otimes_{A_{j(w^-)}}^{\mathbb{L}} M_{w^-}\right) \right) ,
	\end{split}
\end{equation}
where $\mathrm{Cyc}(Q)$ is the set of cycles in the quiver and $\mathrm{fPath}$ is the set of paths ending on framed vertices of the quiver. Here, the notation $\cdots\otimes_{A_{t(e_n)}}$ with noting on the right means tensor product with the first factor on the left.
\end{theorem}

There is a more concise way to express our above result in terms of stratified factorization algebra \cite{Ayala2014FactorizationHO}. In fact, a cycle $\ell  = (e_1,\dots,e_m) \in \mathrm{Cyc}(Q)$ gives us a set of associative algebras $A_{s(e_i)}$ and $A_{s(e_i)}-A_{t(e_i)}$ bimodules $M_{e_i}$. This data $(A_{s(e_i)},M_{e_i})$ defines a stratified factorization algebra $\mathcal{A}_\ell$ on $S^1$ with $m$ marked points. Similarly, a path $\rho \in  \mathrm{fPath}(Q)$ defines a stratified factorization algebra $\mathcal{A}_\rho$ on $[0,1]$. Then we can rewrite \ref{main_result} (Theorem \ref{thm_main2}) as follows
\begin{equation*}
		\mathrm{Sym} \left(  \bigoplus_{v \in Q_0}\left( \int_{S^1}A_v\right)_{S^1}[1] \oplus  \bigoplus_{\ell \in \mathrm{Cyc}(Q)} \left( \int_{S^1}\mathcal{A}_{\ell}\right) _{\Z_{\deg \ell}}\oplus \bigoplus_{\rho  \in \mathrm{fPath}(Q^{\mathrm{fr}})} \left( \int_{[0,1]}\mathcal{A}_{\rho}\right) \right). 
\end{equation*}

\noindent\textbf{Acknowledgment.} I would like to thank Kevin Costello, Zhengping Gui, Si Li, Lukas Mueller, Nuo Xu for illuminating discussions. Research at Perimeter Institute is supported in part by the Government of Canada through the Department of Innovation, Science and Economic Development Canada and by the Province of Ontario through the Ministry of Colleges and Universities.

\section{Notation and convention}

\begin{itemize}
	\item Through out we let $k$ be a field of characteristic zero. Without otherwise mentioned, $\otimes$ refers to tensor product over $k$.
	\item Let $G$ be a group and $V$ a left $k[G]$ module. We denote the space of invariants as
	\begin{equation*}
		V^G := \{v \in V\mid g.v = v \text{ for all }g \in G \},
	\end{equation*}
	and the space of coinvariants as
	\begin{equation*}
		V_G := k\otimes_{k[G]}V = V/\{v  - g.v \mid v\in V,\; g \in G \}.
	\end{equation*}
	\item A ($\Z$-) graded vector space is a vector space $V$ together with a decomposition into a direct sum $V = \bigoplus\limits_{i \in \Z}V_i$. For $k\in \Z$, we use the $k$-th shift notation $V[k]$ as follows
	\begin{equation*}
		(V[k])_i = V_{i - k}
	\end{equation*}
	\item For an associative algebra $A$, we denote $A^{\text{op}}$ the opposite algebra of $A$. We also denote $A^e = A\otimes A^{\text{op}}$. 
	\item Let $\mathrm{Mat}_{N}$ denote the space of $N\times N$ matrices over $k$. A basis of  $\mathrm{Mat}_{N}$ is given by $\{E_{ij}\}_{1\leq i,j \leq N}$, where $E_{ij}$ is the matrix with $1$ in the $i$-th row and $j$-th column and $0$ everywhere else. We also denote $\mathrm{Mat}_{N}(A) = \mathrm{Mat}_{N}\otimes A$ and $E_{ij}^a := E_{ij}\otimes a \in \mathrm{Mat}_{N}(A) $.
	\item Suppose $A$ is a $S$ bimodule. We use the notation $A\otimes_S := A\otimes_{S^e}S = A/[A,S] $. More generally, if $A_1$ is a left $S$ module and $A_n$ is a right $S$ module, then we use the notation $A_1\otimes \dots A_n\otimes_S := A_1\otimes \dots A_n\otimes_{S^e}S$. This gives $ A_1\otimes \dots A_n$ factored by the relation $(sa_1,\dots ,a_n) = (a_1,\dots,a_ns)$.
\end{itemize}

\section{Preliminaries}

\subsection{Fundamental Theorems of Invariant Theory}
In this section, we briefly review some classical results on invariant theory. We refer to \cite{loday1992cyclic} for more detail.

Let $V$ be a finite dimensional $k$-vector space and let $GL(V)$ be its group of automorphisms. The left action of the symmetric group $S_n$ on $V^{\otimes n}$ is given by place permutation
\begin{equation*}
\sigma (v_1,\dots,v_n) = (v_{\sigma^{-1}(1)} ,\dots,v_{\sigma^{-1}(n)} ).
\end{equation*}
The above map gives rise to a ring homomorphism
\begin{equation*}
\mu : k[S_n] \to \mathrm{End}(V^{\otimes n}).
\end{equation*}
Let $GL(V)$ act on $V^{\otimes n}$ diagonally $\alpha (v_1,\dots,v_n) = (\alpha v_1,\dots,\alpha v_n)$, and act on $\mathrm{End}(V^{\otimes n})$ by conjugation $\alpha(f) = \alpha \circ f \circ \alpha^{-1}$. We find that the image of $\mu$ lies in the invariant algebra $\mathrm{End}(V^{\otimes n})^{GL(V)}$. Therefore, we have the ring homomorphism:
\begin{equation*}
\mu : k[S_n] \to \mathrm{End}(V^{\otimes n})^{GL(V)}.
\end{equation*}

We have the following theorem from classical invariant theory.

\begin{theorem}
	Let $V$ be a $k$-vector space with $\dim V = r \geq n$. Then $\mu : k[S_n] \to \mathrm{End}(V^{\otimes n})^{GL(V)}$ is an isomorphism.
\end{theorem}

It will be useful to translate the fundamental theorems of invariant theory into the coinvariant framework. The natural pairing on $V^*\otimes V$ introduces a trace map on $\mathrm{End}(V)$, $ \text{Tr}:\;\mathrm{End}(V) \overset{\approx}{\longrightarrow} V^*\otimes V \overset{\langle-,-\rangle}{\longrightarrow} k$. This trace map induces the nondegenerate pairing:
\begin{equation*}
\begin{aligned}
\mathrm{End}(V) \otimes \mathrm{End}(V) \to k;\\
\alpha\otimes\beta \mapsto \text{Tr}(\alpha\beta).
\end{aligned}
\end{equation*}
It follows immediately that $\mathrm{End}(V)^{\otimes n} \approx (\mathrm{End}(V)^{\otimes n})^*$.

Note that for a $G$ module $W$, we have an isomorphism $(W_G)^* \cong (W^*)^G$. Specifically, we consider the $GL(V)$ module $\mathrm{End}(V^{\otimes n})$. Then we have
\begin{equation*} 
\mathrm{End}(V^{\otimes n})^{GL(V)} \cong (\mathrm{End}(V)^{\otimes n})^{GL(V)}\cong ((\mathrm{End}(V)^{\otimes n})^*)^{GL(V)} \cong ((\mathrm{End}(V)^{\otimes n})_{GL(V)})^*.
\end{equation*}
This determines a map
\begin{equation*}
T: k[S_n] \to \mathrm{End}(V^{\otimes n})^{GL(V)} \cong ((\mathrm{End}(V)^{\otimes n})_{GL(V)})^*.
\end{equation*}
\begin{prop}
	\label{Prop_T}
	Let $\sigma = (\mu_1,\dots,\mu_j)(\nu_1,\dots,\nu_k) \dots (\rho_1,\dots,\rho_l)$ be the cycle decomposition of the permutation $\sigma$. Then the map $T$ defined above is explicitly given by
	\begin{equation}\label{T_inv}
	T(\sigma)(\alpha_1,\dots,\alpha_n) = \Tr(\alpha_{\mu_1}\cdots\alpha_{\mu_j})\Tr(\alpha_{\nu_1}\cdots\alpha_{\nu_k})\cdots\Tr(\alpha_{\rho_1}\cdots\alpha_{\rho_l}).
	\end{equation}
\end{prop}

The $k$-module $k[S_n]$ is canonically isomorphic to its dual $k[S_n]^*$ since it comes with a preferred pairing $(\sigma,\tau) = \delta_{\sigma\tau}$. Thus, we have a dual map
\begin{equation*}
T^* : (\mathrm{End}(V)^{\otimes n})_{GL(V)} \to k[S_n].
\end{equation*}
It is clear that we have
\begin{equation*}
T^*(\alpha_1,\dots,\alpha_n) = \sum_{\sigma \in S_n} T(\sigma)(\alpha_1,\dots,\alpha_n) \sigma,
\end{equation*}
where $T(\sigma)(\alpha_1,\dots,\alpha_n)$ is given by the formula \ref{T_inv}. More explicitly, we have 
\begin{prop}\label{Prop_TT}
	Let $\omega$ and $\sigma$ be two permutations in $S_n$. The map $T^*$ satisfies
	\begin{equation*}
		T^*(E_{\omega(1)\sigma(1)},E_{\omega(2)\sigma(2)},\dots,E_{\omega(n)\sigma(n)}) = \sigma\circ\omega^{-1}.
	\end{equation*}
\end{prop}

As an associate algebra, $\mathrm{End}(V)$ inherits a structure of Lie algebra, which we denote as $\mathfrak{gl}(V)$. The Lie algebra $\mathfrak{gl}(V)$ is a module over itself, and so is the tensor product $\mathfrak{gl}(V)^{\otimes n}$. We have a right action
\begin{equation*}
[(\alpha_1,\dots,\alpha_n),\alpha] = \sum_{i}(\alpha_1,\dots,\alpha_{i-1},[\alpha_i,\alpha],\alpha_{i+1},\dots,\alpha_n).
\end{equation*}

\begin{lem}
	The module of coinvariants $(\mathfrak{gl}(V)^{\otimes n})_{\mathfrak{gl}(V)}$ is equal to the module $(\mathfrak{gl}(V)^{\otimes n})_{GL(V)}$.
\end{lem}
Therefore, for the case  $\dim V = r \geq n$, we have an isomorphism
\begin{equation*}
	(\mathfrak{gl}(V)^{\otimes n})_{\mathfrak{gl}(V)} \cong k[S_n].
\end{equation*}
This isomorphism plays an important role in the LQT theorem and its generalization we present in this paper.

\subsection{Hochschild and Cyclic homology}
In this section, we review aspects of Hochschild and Cyclic homology. 

A bimodule over $A$ is a $k$-module $M$ that is both a left and a right $A$-module for which $(am)b = a(mb)$, for any $a,b\in A$, $m\in M$. We see that $M$ being a $A$ bimodule is equivalent to $M$ being a $A^e$ left module, where the action is defined by $(a\otimes b) m = amb$. The algebra $A$ itself is a $A^e$-module given by the left and right multiplication: $(a\otimes b) c = acb$. 

\begin{defn}
	Let $M$ be an $A$ bimodule. The Hochschild homology of $A$ with coefficient in $M$,	denoted $H_{\sbullet}(A,M)$,  is defined as 	
\begin{equation}
	\mathrm{Tor}_{\sbullet}^{A^e}(A,M) = H_{\sbullet}(M\otimes_{A^e}^{\mathbb{L}}A).
\end{equation}

For $M = A$, $H_{\sbullet}(A,A)$ is also called the Hochschild homology of $A$. We denote it as $ H_{\sbullet}(A)$.
\end{defn}
We use the standard bar construction to compute the Hochschild homology.
\begin{defn}
	Let $M$ be an $A$ bimodule.  The Hochschild Bar complex of $A$ with coefficient in $M$ is defined to be $B_{\sbullet}(A,M) = M\otimes (\oplus_{n\geq 0}(A[1])^{\otimes n})$. We define the following maps $B_{n}(A,M) \to B_{n-1}(A,M)$:
\begin{equation*}
	\begin{split}
				d_1(m\otimes a_1\otimes\dots \otimes a_{n}) &= ma_1\otimes a_2\otimes \dots \otimes a_n,\\
		d_i(m\otimes a_1\otimes\dots \otimes a_{n}) &= m\otimes \dots \otimes a_{i}a_{i+1}\otimes \dots \otimes a_n,\\
		d_n(m\otimes a_1\otimes\dots \otimes a_{n}) &= a_nm\otimes a_1\otimes \dots \otimes a_{n-1}.
	\end{split}
\end{equation*}
Then the Hochschild differential $b: B_n(A,M) \to B_{n - 1}(A,M)$ is given by  $b = \sum_{i = 0}^n (-1)^id_i$.
\end{defn}

\begin{prop}
	We have the following isomorphism 
	\begin{equation*}
		H_{\sbullet}(A,M) \cong H_{\sbullet}(B_{\sbullet}(A,M),b).
	\end{equation*}
\end{prop}

It will be convenient for us to use a relative version of the Bar complex  to compute the Hochschild homology. Let $S$ be a subring of $A$. We define the Bar complex over $S$ as follows:
\begin{equation*}
	B^S_{\sbullet}(A,M) : = \bigoplus_{n \geq 0} M\otimes_S\underbrace{A[1]\otimes_SA[1]\otimes_S\dots A[1]}_{n}\otimes_S.
\end{equation*}
\begin{remark}
	Recall that the notation $M\otimes_S A\otimes_S\dots A\otimes_S$ means $(M\otimes_S A\otimes_S\dots A)\otimes_{S^e}S$.
\end{remark}
The Hochschild differential $b$ is compatible with this definition so there is a well-defined complex $(B^S_{\sbullet}(A,M) , b)$. Denote its $n$-th homology by $HH^S_n(A,M)$. We are interested in the cases when $HH^S_n(A,M)$ is isomorphic to the Hochschild homology.

\begin{defn}
	An unital $k$-algebra $S$ is called separable if there exists an element 
	$$
	e = \sum u_i\otimes v_i \in S\otimes S^{op},
	$$
	such that $\sum u_i v_i = 1$ and $(s\otimes 1)e = (1\otimes s) e$ for any $s\in S$. Such an element $e$ is also called the idempotent.
\end{defn}

We have the following slight generalization of \cite{loday1992cyclic} Theorem 1.2.13.
\begin{theorem}
	\label{Thm:Hoch_S}
	Let $S$ be a separable $k$ algebra. Then for any unital $S$-algebra $A$ there is a canonical isomorphism
	\begin{equation*}
HH_n(A,M)\cong HH^S_n(A,M).
	\end{equation*}
\end{theorem}
\begin{proof}
	There is an canonical projection $\phi: B_n(A,M) \to B^S_n(A,M) $.
	
	Using the idempotent $e = \sum u_i\otimes v_i$ one can construct a map $\psi : B^S_n(A,M)  \to B_n(A,M) $:
	\begin{equation}
		\label{spliting_psi}
		\psi(m\otimes_Sa_1\otimes_S\dots a_n\otimes_S) =  \sum_{j_0,j_1,\dots,j_n}v_{j_{n}} m u_{j_0}\otimes v_{j_0} a_1 u_{j_1}\otimes \dots \otimes v_{j_{n-1}} a_n u_{j_n}.
	\end{equation}
It is easy to check that $\phi\circ\psi = id$. On the other hand, we construct a pre-simplicial homotopy between $\psi\circ\phi $ and $id$. Let 
\begin{equation}
	h_i(m\otimes a_1\otimes \dots a_n) = \sum_{j_0,j_1,\dots,j_i}m u_{j_0}\otimes v_{j_0} a_1 u_{j_1}\otimes \dots v_{j_{i-1}} a_{i} u_{j_i}\otimes v_{j_i} \otimes a_{i+1}\otimes \dots a_n .
\end{equation}
We can check that $d_0h_0(m\otimes a_1\otimes \dots a_n) = d_0(\sum_{j}m u_{j}\otimes v_{j} \otimes a_1 \dots a_n) = m\otimes a_1\otimes \dots a_n$. Therefore $d_0h_0 = id$. We can also check that
\begin{equation*}
	\begin{aligned}
		d_{n+1}h_n(m\otimes a_1\otimes \dots a_n)& = d_{n+1}(\sum_{j_0,j_1,\dots,j_n} m u_{j_0}\otimes v_{j_0} a_1 u_{j_1}\otimes \dots \otimes v_{j_{n-1}} a_n u_{j_n}\otimes v_{j_{n}}) \\
		&=\psi\circ \phi(m\otimes a_1\otimes \dots a_n).
	\end{aligned}
\end{equation*} 
Therefore $d_{n+1}h_n = \psi\circ \phi$. It is also immediate to check the remaining relation
\begin{equation}
	d_i h_j = \begin{cases}
		h_jd_i, & \text{ for } i< j,\\
		d_ih_{i - 1}, & \text{ for } i = j \neq 0,\\
		h_{j}d_{i - 1}, & \text{ for } i > j+1\,.
	\end{cases}
\end{equation}
Let $h = \sum_{i = 0}^n(-1)^{i}h_i$. Then $hd + dh = id - \psi\circ \phi $.
\end{proof}

For $L$ a left $A$ module and $M$ a right $A$ module, $L\otimes M$ has a canonical $A$ bimodule structure. Note that
\begin{equation*}
	(L\otimes M)\otimes_{A^e}^{\mathbb{L}}A \cong M\otimes_A^{\mathbb{L}}L.
\end{equation*} 
We have the following
\begin{cor}
	\label{Cor_Tor_S}
	Let $L$ be a left $A$ module and $M$ be a right $A$ module. Let $S$ be a separable subalgebra of $A$. We define the complex
	\begin{equation*}
		B^{S}_{\sbullet}(M,A,L)  = \bigoplus_{n\geq 0 }N\otimes_S(A[1])^{\otimes_S n}\otimes_SL.
	\end{equation*}
Then there is a canonical isomorphism
\begin{equation*}
	H_n(B^{S}_{\sbullet}(M,A,L)) = \mathrm{Tor}_n^A(M,L).
\end{equation*}
\end{cor}

Now we introduce the cyclic homology. The cyclic group $\mathbb{Z}_{n+1}$ acts on $B_{n}(A,A)$ via permuting the tensor factor. Explicitly, let $t$ be the generator of $\mathbb{Z}_{n+1}$. Then its action on $B_{n}(A,A)$ is given by
$$
t(a_0,\dots,a_n) = (-1)^{n}(a_n,a_0,\dots,a_{n-1}).
$$
It is called the Connes' cyclic operator. Let $N = 1+t = t^2 + \dots +t^n$ be the corresponding norm operator.

The cyclic bicomplex $CC_{\sbullet,\sbullet}(A)$ is defined as follows
\begin{center}
	\begin{tikzcd}
		~\arrow[d]&~\arrow[d]&~\arrow[d]&~\\
		A^{\otimes 3} \arrow[d,"b"'] & A^{\otimes 3} \arrow[d,"-b'"'] \arrow[l,"1-t"'] &A^{\otimes 3} \arrow[d,"b"']\arrow[l,"N"'] &~\arrow[l,"1-t"'] \\
		A^{\otimes 2} \arrow[d,"b"'] & A^{\otimes 2} \arrow[d,"-b'"'] \arrow[l,"1-t"'] &A^{\otimes 2} \arrow[d,"b"']\arrow[l,"N"']&~\arrow[l,"1-t"']  \\
		A  &A\arrow[l,"1-t"'] &A\arrow[l,"N"'] &~\arrow[l,"1-t"'] 
	\end{tikzcd}
\end{center}
The differential $b$ is the Hochschild differential and the differential $b'$ is defined as $b' = \sum_{i = 0}^{n-1} (-1)^id_i$.
One can check from the definition that all the squares commute and the above is indeed a bicomplex.

\begin{defn}
	The cyclic homology $HC_{n}(A)$ of the associative algebra $A$ is defined as the homology of the total complex of $CC_{\sbullet,\sbullet}(A)$
	\begin{equation*}
		HC_{n}(A) = H_n(\mathrm{Tot}\,CC_{\sbullet,\sbullet}(A)).
	\end{equation*}
\end{defn}

Consider the space of coinvariants of $B_{n}(A,A)$ under the action of $\Z_{n+1}$. We denote this space by $C^{\lambda}_n(A) = (A^{\otimes (n+1)})_{\Z_{n+1}}$. We have a well-defined complex
$$
C^{\lambda}_{\bullet}(A): \;\;\dots\overset{b}{\longrightarrow} C^{\lambda}_n(A)\overset{b}{\longrightarrow} C^{\lambda}_{n-1}(A)\overset{b}{\longrightarrow}\dots \overset{b}{\longrightarrow}C^{\lambda}_{0}(A).
$$
called the Connes complex. Its $n$-th homology group is denoted $H^{\lambda}_n(A)$. We have the following result
\begin{theorem}\label{Thm_Connes_cpx}
	Let $k$ be a field of characteristic $0$, then we have an isomorphism
$$
H^{\lambda}_n(A) \cong HC_n(A).
$$
\end{theorem}
Similar to Theorem \ref{Thm:Hoch_S}, we can also consider the cyclic bicomplex and Connes' complex over a ground ring $S$.
\begin{lem}
	Let $S$ be a separable subalgebra of $A$. We define the cyclic bicomplex $CC^S_{\sbullet,\sbullet}(A)$ over $S$ as follows
	\begin{center}
		\begin{tikzcd}
			~\arrow[d]&~\arrow[d]&~\arrow[d]&~\\
			A\otimes_S A\otimes_SA\otimes_S \arrow[d,"b"'] & A\otimes_S A\otimes_SA\otimes_S \arrow[d,"-b'"'] \arrow[l,"1-t"'] &A\otimes_S A\otimes_SA\otimes_S \arrow[d,"b"']\arrow[l,"N"'] &~\arrow[l,"1-t"'] \\
			A\otimes_S A\otimes_S \arrow[d,"b"'] & A\otimes_S A\otimes_S \arrow[d,"-b'"'] \arrow[l,"1-t"'] &A\otimes_S A\otimes_S \arrow[d,"b"']\arrow[l,"N"']&~\arrow[l,"1-t"']  \\
			A\otimes_S  &A\otimes_S\arrow[l,"1-t"'] &A\otimes_S\arrow[l,"N"'] &~\arrow[l,"1-t"'] 
		\end{tikzcd}
	\end{center}
Then we have an isomorphism
\begin{equation}\label{CC_S_CH}
	H_n(\mathrm{Tot}\,CC^S_{\sbullet,\sbullet}(A)) \cong HC_{n}(A).
\end{equation}
\end{lem}
\begin{proof}
	The projection map $\phi: B_n(A,A) \to B_n^S(A,A)$ and the map $\psi:B_n^S(A,A) \to  B_n(A,A)  $ defined by \ref{spliting_psi} are compatible with the cyclic operator. 
$$
		\begin{aligned}
			\psi \circ  t(a_0\otimes_Sa_1\otimes_S\dots a_n\otimes_S)	& =  (-1)^n\sum_{j_0,j_1,\dots,j_n}v_{j_{n}} a_n u_{j_0}\otimes v_{j_0} a_0 u_{j_1}\otimes \dots \otimes v_{j_{n-1}} a_{n-1} u_{j_n}\\
			& = t\circ\psi(a_0\otimes_Sa_1\otimes_S\dots a_n\otimes_S) .
		\end{aligned}
$$
As a result, they define morphisms between the bicomplex $CC_{\sbullet,\sbullet}(A)$ and $CC^S_{\sbullet,\sbullet}(A)$. Using Theorem \ref{Thm:Hoch_S} we see that they are quasi-isomorphisms when restrict to the $b$-columns. All the $b'$-columns are acyclic. Therefore $\phi$ and $\psi$ are also quasi-isomorphism when restricted to the $b'$-columns. So the total complexes are quasi-isomorphic.
\end{proof}

\begin{cor}
	\label{Cor_cyc}
	Let $S$ be a separable subalgebra of $A$. We consider the Connes' complex over $S$:
	\begin{equation}
		C^{\lambda,S}_{n}(A)  = (\underbrace{A\otimes_SA\otimes_S\dots A\otimes_S}_{n+1})_{\Z_{n+1}},
	\end{equation}
	with the differential $b$. Then there is a canonical isomorphism
	\begin{equation}
		H_n(C^{\lambda,S}_{\sbullet}(A)) = HC_{n}(A).
	\end{equation}
\end{cor}
\begin{proof}
	Using the isomorphism \ref{CC_S_CH}, we only need to prove $	H_n(C^{\lambda,S}_{\sbullet}(A)) \cong H_n(\mathrm{Tot}\,CC^S_{\sbullet,\sbullet}(A)) $.
	This proof is the same as the proof of Theorem \ref{Thm_Connes_cpx}, which we recall now. We construct homotopy from $\mathrm{id}$ to $0$ for each row complex. Let 
	\begin{equation*}
		h' := \frac{1}{n+1} \mathrm{id},\quad h = - \frac{1}{n+1}\sum_{i = 1}it^i
	\end{equation*}
be maps from $B_{n}^S(A)$ to itself. One can verifies that
\begin{equation*}
	h'N + (1-t)h = \mathrm{id},\quad Nh' + h(1-t) = \mathrm{id}.
\end{equation*}
Therefore, each row of the bicomplex $CC^S_{\sbullet,\sbullet}(A)$ has homology concentrated in degree $0$ and with $H_0(CC^S_{\sbullet,n}(A),d_v) = C^{\lambda,S}_{n}(A)$.
\end{proof}

\subsection{Loday-Quillen-Tsygan theorem}
Let $\mathfrak{gl}_N(A) = A\otimes \mathfrak{gl}_N$ be the Lie algebra of matrices valued in $A$. For each $N$, we have a natural inclusion of Lie algebra $\mathfrak{gl}_N(A) \to \mathfrak{gl}_{N+1}(A)$ given by $\alpha \to \begin{pmatrix}
\alpha&0\\0&0
\end{pmatrix}$. With respect to these inclusions we define direct limit $$\mathfrak{gl}(A) = \lim_{\overset{\longrightarrow}{N}}\mathfrak{gl}_N(A),$$ and consider the Lie algebra homology $$H_{\sbullet}(\mathfrak{gl}(A)) = \lim_{\overset{\longrightarrow}{N}} H_{\sbullet}(\mathfrak{gl}_N(A)).$$

In this section, we sketch the proof of the Loday-Quillen-Tsygan theorem. We refer to \cite{loday1992cyclic} for the full detail.
\begin{theorem}
	There is a canonical isomorphism of graded Hopf algebras
	$$
	H_{\bullet}(\mathfrak{gl}(A)) \cong \mathrm{Sym} (HC_{\bullet }(A)[1]).
	$$
\end{theorem}

\begin{proof}(sketch)
	Step 1. We can decompose the \CE complex by the adjoint action of $\mathfrak{gl}_N$, and only consider the trivial representation in the homology. This is due to the following proposition.
	\begin{prop}
		\label{Lie_ho_red}
		Let $\mathfrak{g}$ be a Lie algebra and $V$ a $\mathfrak{g}$ module. Let $\mathfrak{h}$ be a reductive sub-Lie algebra of $\mathfrak{g}$. Then the surjective map $(V\otimes \mathrm{Sym} (\mathfrak{g}[1])) \to(V\otimes \mathrm{Sym} (\mathfrak{g}[1]))_{\mathfrak{h}}$ induces an isomorphism on homology
$$
			H_{\sbullet}(\mathfrak{g},V) \cong H_{\sbullet}((V\otimes \mathrm{Sym} (\mathfrak{g}[1]))_{\mathfrak{h}},d_{CE})\,. 
$$
	\end{prop}
	
	In our case, we get a quasi-isomorphism of complex
$$
		(\mathrm{Sym}(\mathfrak{gl}_N(A)[1]),d_{CE}) \to (\mathrm{Sym}(\mathfrak{gl}_N(A)[1])_{\mathfrak{gl}_N},d_{CE}).
$$
	
	Step 2. Note that $\mathrm{Sym}^n(\mathfrak{gl}_N(A)[1])$ can be regarded as the coinvariant of $\mathfrak{gl}_N(A)^{\otimes n}$ under the signed action of $S_n$. The action of $S_n$ commutes with the adjoint action of $\mathfrak{gl}_N$. We have
	$$
	(\mathrm{Sym}^n(\mathfrak{gl}_N(A)[1]))_{\mathfrak{gl}_N} \cong ((\mathfrak{gl}_N^{\otimes n} \otimes (A[1])^{\otimes n})_{\mathfrak{gl}_N})_{S_n} \cong ((\mathfrak{gl}_N^{\otimes n} )_{\mathfrak{gl}_N}\otimes (A[1])^{\otimes n})_{S_n} .
	$$
	Classical invariant theory tells us that for $N\geq n$, there is an isomorphism 
	$$
	T^*: (\mathfrak{gl}_N^{\otimes n} )_{\mathfrak{gl}_N} \to k[S_n].
	$$
	As a consequence, we have an isomorphism
	$$
	\Theta: (\mathrm{Sym}^n(\mathfrak{gl}(A)[1]))_{\mathfrak{gl}} \cong (k[S_n]\otimes (A[1])^{\otimes n})_{S_n} .
	$$
	According to Proposition \ref{Prop_TT}, this isomorphism $\Theta$ is given as follows:
\begin{equation*}
	\Theta(E^{a_1}_{1\sigma(1)}\wedge E^{a_2}_{2\sigma(2)} \dots \wedge E^{a_n}_{n\sigma(n)}) = (\sigma\otimes (a_1, a_2,\dots a_n)) \in (k[S_n]\otimes (A[1])^{\otimes n})_{S_n} .
\end{equation*}
	Under this isomorphism, $S_n$ acts on $k[S_n]$ by conjugation and acts on $A^{\otimes n}$ by signed permutation.
	
	Step 3. The diagonal map $\Delta:\mathfrak{gl}(A)\to \mathfrak{gl}(A)\times\mathfrak{gl}(A)$ induces a map on the \CE complex $\Delta: (\mathrm{Sym}^n(\mathfrak{gl}(A)[1]))_{\mathfrak{gl}} \to (\mathrm{Sym}^n(\mathfrak{gl}(A)[1]))_{\mathfrak{gl}}\otimes(\mathrm{Sym}^n(\mathfrak{gl}(A)[1]))_{\mathfrak{gl}}$. This map defines a coassociative and cocommutative coproduct on the complex $(\mathrm{Sym}(\mathfrak{gl}(A)[1]))_{\mathfrak{gl}}$.

	There is also a map $\oplus: \mathfrak{gl}(A)\times \mathfrak{gl}(A) \to \mathfrak{gl}(A)$ which can be schematically described as
\begin{equation}
	\label{marix_dsum}
		\begin{pmatrix}
		\star&\star&\star &\cdots\\
		\star&\star&\star &\cdots\\
		\star&\star&\star &\cdots\\
		\vdots&\vdots&\vdots&\vdots
	\end{pmatrix}\oplus \begin{pmatrix}
		\times&\times&\times &\cdots\\
		\times&\times&\times &\cdots\\
		\times&\times&\times &\cdots\\
		\vdots&\vdots&\vdots&\vdots
	\end{pmatrix} =  
	\begin{pmatrix}
		\star&0&\star&0&\star&0 &\cdots\\
		0&\times&0&\times&0&\times &\cdots\\
		\star&0&\star&0&\star&0 &\cdots\\
		0&\times&0&\times&0&\times &\cdots\\
		\vdots&\vdots&\vdots&\vdots&\vdots&\vdots&\vdots
	\end{pmatrix}.
\end{equation}
	It induces a map of complex $ \mathrm{Sym}(\mathfrak{gl}(A)[1])\otimes \mathrm{Sym}(\mathfrak{gl}(A)[1])\to \mathrm{Sym}(\mathfrak{gl}(A)[1])$. Although this map is neither associative nor commutative, it induces a graded associative and commutative product on the $\mathfrak{gl}$ coinvariant complex  $$\oplus_*: (\mathrm{Sym}(\mathfrak{gl}(A)[1]))_{\mathfrak{gl}}\otimes (\mathrm{Sym}(\mathfrak{gl}(A)[1]))_{\mathfrak{gl}} \to (\mathrm{Sym}(\mathfrak{gl}(A)[1]))_{\mathfrak{gl}}.$$ Therefore, $((\mathrm{Sym}(\mathfrak{gl}(A)[1]))_{\mathfrak{gl}},d_{CE})$ equipped with $\Delta$ and $\oplus_*$ is a commutative and cocommutative differential graded Hopf algebra. We have the following theorem
	\begin{theorem}
		Let $\mathcal{H}$ be a cocommutative differential graded Hopf algebra over a field of characteristic $0$. We have that the homology and primitive functors commute,
		\begin{equation*}
			H(\mathrm{Prim}\mathcal{H}) = \mathrm{Prim}H(\mathcal{H}).
		\end{equation*}
	\end{theorem}

\begin{theorem}
	Let $\mathcal{H}$ be a connected cocommutative graded Hopf (resp. differential graded Hopf) algebra over a field of characteristic $0$. Then the canonical map 
	\begin{equation*}
		U(\mathrm{Prim}\mathcal{H}) \to \mathcal{H}
	\end{equation*}
is an isomorphism of graded Hopf (resp. differential graded Hopf) algebras.
\end{theorem}
	
	As a corollary, $H_{\bullet}(\mathfrak{gl}(A))$ is a commutative and cocommutative graded Hopf algebra whose primitive part is the homology of $(\text{Prim}(\mathrm{Sym}(\mathfrak{gl}(A)[1]))_{\mathfrak{gl}},d_{CE})$. Moreover, since $H_{\bullet}(\mathfrak{gl}(A))$ is commutative, it is isomorphic to the graded symmetric algebra of $H_{\sbullet}(\mathrm{Prim}(\mathrm{Sym}(\mathfrak{gl}(A)[1])_{\mathfrak{gl}}))$.
	
	Step 4. We need to determine the primative part of $L_\bullet := \oplus_{n\geq 1}(k[S_n]\otimes (A[1])^{\otimes n})_{S_n} $. The coproduct on $\mathrm{Sym}(\mathfrak{gl}(A)[1])_{\mathfrak{gl}}$ induces a coproduct on $L_{\bullet}$ which we also denote as $\Delta$.
	\begin{equation}\label{Co_pro}
			\Delta(\sigma\otimes(\alpha_1,\dots \alpha_n))= \sum_{I,J}(\sigma_I\otimes (\alpha_{i_1},\alpha_{i_2}\dots)) \otimes (\sigma_J\otimes (\alpha_{j_1},\alpha_{j_2}\dots)) ,
	\end{equation}
	where the sum go through all ordered partitions $(I,J)$ of $S_n$ such that $\sigma(I) = I,\sigma(J) = J$.
	
	Let $U_n\subset S_n$ be the conjugacy class of the cycle $(1,2,\dots,n)$. We see that $P_{\bullet} :=  \oplus_{n\geq 1}(k[U_n]\otimes (A[1])^{\otimes n})_{S_n}$ is a subspace of $\mathrm{Prim}L_{\bullet}$. Moreover, any permutation is a product of cycles. We thus have a surjection $\mathrm{Sym} P_{\bullet} \to L_{\bullet}$. This tell us that $P_{\bullet} = \text{Prim}L_{\bullet}$. Therefore, we have an isomorphism
	$$
	\text{Prim}(\mathrm{Sym}(\mathfrak{gl}(A)[1]))_{\mathfrak{gl}} \cong \oplus_{n\geq 1}(k[U_n]\otimes (A[1])^{\otimes n})_{S_n}.
	$$
	As a $S_n$ module, $k[U_n]$ is isomorphic to the induced representation of the trivial representation of $\mathbb{Z}_n$ to $S_n$, $\mathrm{Ind}_{\Z_n}^{S_n}k$. We have
\begin{equation}
	\label{Prim_to_cyc}
		(\mathrm{Ind}_{\mathbb{Z}_n}^{S_n}k \otimes (A[1])^{\otimes n})_{S_n} \cong ((A[1])^{\otimes n})_{\Z_n} = C_{n-1}^{\lambda}(A).    
\end{equation}
	We see that $\text{Prim}(\mathrm{Sym}(\mathfrak{gl}(A)[1]))_{\mathfrak{gl}} \cong C_{\bullet - 1}^{\lambda}(A)$. One can also verify that the \CE differential $d_{CE}$ induces the Hochschild differential $b$.
\end{proof}

\section{Homology of Lie algebra of matrices associated to quiver}

\subsection{Adding fundamental representation}
In this section, we consider a generalization of the LQT theorem by considering Lie algebra homology with coefficient in the symmetric powers of the fundamental and the dual of fundamental representation of $\mathfrak{gl}_N(A)$. While we believe that the results in this section have not appeared in the literature, similar statement should already be known to some experts (c.f. \cite{Costello:2018zrm}).

Let $L$ be a left $A$ module and $M$ a right $A$ module. Let $\C^N$ be the fundamental representation of $\mathfrak{gl}_N$. 
Then $\C^N(L) := \C^N\otimes L $ is a $\mathfrak{gl}_N(A)$ module given by the following
\begin{equation}\label{fund_mod}
	(A \otimes a)(v \otimes x) = (Av \otimes ax).
\end{equation}
Similarly, for the dual representation $(\C^N)^*$, $(\C^N)^*(M):= (\C^N)^*\otimes M$ is a $\mathfrak{gl}_N(A)$ module given by 
\begin{equation}\label{afund_mod}
	(A \otimes a)(v^* \otimes x) = (-A^tv^* \otimes xa).
\end{equation}

We consider the following super Lie algebra
\begin{equation*}
	\mathcal{L}_N : = \mathfrak{gl}_N(A)\ltimes (\C^N(L) \oplus  (\C^N)^*(M))[-1].
\end{equation*}
The natural inclusion $\mathcal{L}_N \to \mathcal{L}_{N+1}$ defines direct limit $\mathcal{L}: =  \lim\limits_{\overset{\longrightarrow}{N}}  \mathcal{L}_N$. We have
\begin{equation*}
	\mathcal{L} = \mathfrak{gl}(A)\ltimes (\mathcal{V}(L) \oplus  \mathcal{V}^*(M))[-1],
\end{equation*}
where $\mathcal{V} =  \lim\limits_{\overset{\longrightarrow}{N}}  \C^N$. In this section, we consider the Lie algebra homology
\begin{equation*}
	H_{\sbullet}(\mathcal{L})  =  \lim_{\overset{\longrightarrow}{N}}  H_{\sbullet}(\mathcal{L}_N).
\end{equation*}
 By the definition, there is a natural isomorphism between the Lie algebra homology of $\mathcal{L}$ and the Lie algebra homology of $\mathfrak{gl}(A)$ valued in $\mathrm{Sym}(\mathcal{V}(L) \oplus  \mathcal{V}^*(M))$
\begin{equation*}
	H_{\sbullet}(\mathcal{L})  = H_{\sbullet}(\mathfrak{gl}(A),\mathrm{Sym}(\mathcal{V}(L) \oplus  \mathcal{V}^*(M))).
\end{equation*}

\begin{theorem}
	\label{LQT_addf}
	There is a canonical isomorphism of graded Hopf algebras
\begin{equation*}
	H_{\sbullet}(\mathfrak{gl}(A),\mathrm{Sym}(\mathcal{V}(L) \oplus  \mathcal{V}^*(M))) \cong \mathrm{Sym}(HC_{\sbullet}(A)[1] \oplus \mathrm{Tor}^A_{\sbullet}(M,L) ).
\end{equation*}
\end{theorem}
\begin{proof}

Step 1. Applying Proposition \ref{Lie_ho_red}. We only need to compute the $\mathfrak{gl}_N(k)$ coinvariant of the \CE complex $\mathrm{Sym}(\mathcal{L}_N[1])_{\mathfrak{gl}_N}$.

Step 2. Application of invariant theory. We observe that 
$$
(\mathrm{Sym}^k(\mathfrak{gl}_N(A)[1])\otimes \mathrm{Sym}^{l}(\C^N(L))\otimes  \mathrm{Sym}^{i}( (\C^N)^*(M)))_{\mathfrak{gl}_N}.
$$
is only non zero when $i=l$. In this case, we have
\begin{align*}
		&\left(\mathrm{Sym}^k(\mathfrak{gl}_N(A)[1]) \otimes \mathrm{Sym}^{l}(\C^N(L))\otimes  \mathrm{Sym}^{l}( (\C^N)^*(M))\right)_{\mathfrak{gl}_N}\\
		= & \left(\left((\mathfrak{gl}_N(A)[1]))^{\otimes k} \otimes (\C^N(L))^{\otimes l} \otimes  ((\C^N)^*(M)) ^{\otimes l} \right)_{S_k\times S_l\times S_l}\right)_{\mathfrak{gl}_N}\\
		= &  \left((\mathfrak{gl}_N^{\otimes k} \otimes (\C^N)^{\otimes l} \otimes(( \C^N)^*)^{\otimes l})_{\mathfrak{gl}} \otimes (A[1]^{\otimes k}\otimes L^{\otimes l} \otimes M^{\otimes l}) \right)_{S_k\times S_l\times S_l}\\
		= & \left((\mathfrak{gl}_N^{\otimes k} \otimes \mathfrak{gl}_N^{\otimes l})_{\mathfrak{gl}_N} \otimes (A[1]^{\otimes k}\otimes L^{\otimes l} \otimes M^{\otimes l}) \right)_{S_k\times S_l\times S_l}.
\end{align*}
Using classical invariant theory, we have an isomorphism 
\begin{equation*}
(\mathfrak{gl}_N^{\otimes k} \otimes \mathfrak{gl}_N^{\otimes l})_{\mathfrak{gl}_N}   \cong \C[S_{k+l}]. 
\end{equation*}
We consider the composite of the above isomorphisms, and denote it as $\widetilde{\Theta}$ at the  $N \to \infty$ limit.
\begin{align*}
		\widetilde{\Theta}: \quad &\left(\mathrm{Sym}^k(\mathfrak{gl}(A)[1]) \otimes \mathrm{Sym}^{l}(\mathcal{V}(L))\otimes  \mathrm{Sym}^{l}(\mathcal{ V}^*(M))\right)_{\mathfrak{gl}} \\
		&\to \left(\C[S_{k+l}] \otimes (A[1]^{\otimes k}\otimes L^{\otimes l} \otimes M^{\otimes l})\right)_{S_k\times S_l\times S_l}.
\end{align*}
Using Proposition \ref{Prop_TT}, we have
\begin{equation}\label{iso_CE_Hoch}
\begin{aligned}
		\widetilde{\Theta}&(E^{a_1}_{1\sigma(1)} \otimes E^{a_2}_{2\sigma(2)} \dots E^{a_k}_{k\sigma(k)}\otimes e^{x_{1}}_{k+1}\dots e^{x_{l}}_{k+l}\otimes f^{y_{1}}_{\sigma(k+1)}\dots f^{y_{l}}_{\sigma(k+l)} )\\
		& = \sigma\otimes ((a_1,\dots,a_k)\otimes(x_1,\dots,x_{l})\otimes (y_1,\dots,y_l)),
\end{aligned}
\end{equation}
where $\{e_i\}$ (resp. $\{f_i\}$) is the canonical basis of  $\mathcal{V}$ (resp. $\mathcal{V}^*$). We also use the notation $e^x_i := e_i\otimes x \in \mathcal{V}\otimes L$ and  $f^y_i := f_i\otimes y \in \mathcal{V}^*\otimes M$.

Using the above isomorphism, we can also identify the action of $S_k\times S_{l}\times S_l$ on $\C[S_{k+l}]$. Let $\rho_{k}: S_k \to S_{k+l}$ the inclusion that sends $\omega \in S_k$ to the element in $S_{k+l}$ that permutes $\{1,2,\dots,k\} \subset \{1,2,\dots,k+l\}$. Then $\omega \in S_k$ acts on $k[S_{k+l}]$ by the conjugate action
\begin{equation*}
	\sigma \in S_{k+l} \mapsto \rho_{k}(\omega)\circ\sigma \circ \rho_{k}(\omega^{-1}).
\end{equation*}

Let $\rho_l:S_l \to S_{k+l}$ be the inclusion that sends $\tau \in S_l$ to the element in $S_{k+l}$ that permutes $\{k+1,k+2,\dots,k+l\}\subset \{1,2,\dots,k+l\}$. Then $(\tau_1,\tau_2) \in S_l\times S_l$ acts on $k[S_{k+l}]$ as follows
\begin{equation*}
	\sigma \in S_{k+l} \mapsto \rho_{k}(\tau_2)\circ\sigma \circ \rho_{k}(\tau_1^{-1}).
\end{equation*}
We have the following proposition characterizing the orbits of $k[S_{k+l}]$ under the $S_l$ action
\begin{prop}\label{Prop_Sl_coinv}
	We define the set $S_{k,l} \subset  S_{k+l} $ as the set of elements whose cycle decomposition has at most one number in $\{k+1,k+2,\dots, k+l\}$ in each cycle.
	
	For any $\sigma \in S_{k+l}$, there exist an $\tau \in S_l$ such that $\sigma \circ \rho_{k}(\tau^{-1}) \in S_{k,l}$. Moreover, such a $\sigma \circ \rho_{k}(\tau^{-1}) \in S_{k,l}$ is the permutation that has the most number of cycles in the orbit of $\sigma \in S_{k+l}$ under the action of $S_l$.
	
\end{prop}
\begin{proof}
	We consider the cycle decomposition of $\sigma$. Let $(l_1, \cdots l_2, \cdots l_m, \cdots)$ be a cycle of $\sigma$, where $l_1,l_2,\dots l_m \in \{k+1,k+2,\dots, k+l\}$ and $\cdots$ refer to elements in $\{1,\dots,k\}$. We define $\rho_{k}(\tau^{-1})$ to act as $l_1\mapsto l_m, l_2 \mapsto l_1,\dots,l_m \mapsto l_{m-1}$. For product of cycles, $\rho_{k}(\tau^{-1})$ is defined similarly for each cycle. Then we can check that $\sigma \circ \rho_{k}(\tau^{-1}) \in S_{k,l}$ has cycle decomposition $(l_1,\cdots)(l_2,\cdots)\dots(l_m,\cdots)$, where $\cdots$ are elements in $\{1,\dots,k\}$. We see that  $\sigma \circ \rho_{k}(\tau^{-1}) \in S_{k,l}$.
\end{proof}

Step 3. Hopf algebra structure.

The diagonal map $\Delta: \mathfrak{gl}(A) \to \mathfrak{gl}(A)\times \mathfrak{gl}(A)$ and $\Delta: \mathcal{V}(L)\oplus \mathcal{V}^*(M) \to( \mathcal{V}(L)\oplus \mathcal{V}^*(M))\times(\mathcal{V}(L)\oplus \mathcal{V}^*(M))$ induce a map of complexes $\mathrm{Sym}(\mathcal{L}[1]) \to \mathrm{Sym}(\mathcal{L}[1]\times \mathcal{L}[1])_{\mathfrak{gl}} $.  Composite this map with the map $\mathrm{Sym}(\mathcal{L}[1]\times \mathcal{L}[1])_{\mathfrak{gl}}  \to \mathrm{Sym}(\mathcal{L}[1])  \otimes\mathrm{Sym}(\mathcal{L}[1])  $ gives us the comultiplication map. One can check that this comultiplication map is coassociative and cocommutative.

In order to define the multiplication map, we consider the direct sum of matrices as in \ref{marix_dsum} and the following maps
$$
	\tilde{\oplus}: \mathcal{V}(M) \times \mathcal{V}(M)  \to \mathcal{V}(M).
$$
given by
\begin{equation}
\begin{pmatrix}
	\star\\ \star \\ \star\\ \vdots
\end{pmatrix}\tilde{\oplus} \begin{pmatrix}
	\times \\ \times\\ \times \\ \vdots
\end{pmatrix}\to 
\begin{pmatrix}
	\star \\\times \\ \star \\ \times \\ \vdots
\end{pmatrix}.
\end{equation}
The map $\tilde{\oplus}: \mathcal{V}^*(N) \times \mathcal{V}^*(N)  \to \mathcal{V}^*(N)$ is defined similarly. These maps induce a product map of the complexes $\mathrm{Sym}(\mathcal{L}[1]) \otimes \mathrm{Sym}(\mathcal{L}[1])  \to \mathrm{Sym}(\mathcal{L}[1]) $, and a map on the $\mathfrak{gl}$ coinvariant complex $\tilde{\oplus}_*: (\mathrm{Sym}(\mathcal{L}[1]))_{\mathfrak{gl}}  \otimes (\mathrm{Sym}(\mathcal{L}[1]) )_{\mathfrak{gl}} \to (\mathrm{Sym}(\mathcal{L}[1]) )_{\mathfrak{gl}}$. We want to show that this product is associative and commutative. From the proof of the original LQT theorem, we know that this map restricts to an associative and commutative product on $ (\mathrm{Sym}(\mathfrak{gl}(A)[1]) )_{\mathfrak{gl}}$. We only need to prove that it also restricts to an associative and commutative product on $ \mathrm{Sym}(\mathcal{V}(L) \oplus \mathcal{V}^*(M)) )_{\mathfrak{gl}}$. This follows from the fact that for any $u,v,w \in \mathcal{V}(M)$, $((u\tilde{\oplus} v )\tilde{\oplus} w)$ and $(u\tilde{\oplus} (v \tilde{\oplus} w))$ only differs by a permutation. Similarly, $u\tilde{\oplus} v $ and $v \tilde{\oplus} u$ only differs by a permutation. We see that $\tilde{\oplus}$ induces an associative and commutative product after we pass to the $\mathfrak{gl}$ coinvariant.

As a result, the Lie algebra homology $H_{\sbullet}(\mathrm{Sym}(\mathcal{L}[1]),d_{CE})$ is isomorphic, as a Hopf algebra, to the symmetric algebra of the homology of the primitive part $\mathrm{Prim} (\mathrm{Sym}(\mathcal{L}[1]) )_{\mathfrak{gl}}$.

Step 4. Computation of the primitive part.
The coalgebra structure on $\C[S_{k+l}] \otimes (A[1]^{\otimes k}\otimes M^{\otimes l} \otimes N^{\otimes l})$ takes a similar form as \ref{Co_pro}. Let
\begin{equation}\label{prim_afund}
	P_{\sbullet} =  \bigoplus_{k \geq 1}(\C[U_{k}] \otimes (A[1]^{\otimes k}))_{S_k} \oplus  (\C[U_{k+1}]\otimes ( A[1]^{\otimes k}\otimes M \otimes N ))_{S_k}.
\end{equation}
We see that $P_{\sbullet}  \subset \mathrm{Prim} \left(\C[S_{k+l}] \otimes (A[1]^{\otimes k}\otimes M^{\otimes l} \otimes N^{\otimes l})\right)_{S_k\times S_l\times S_l}$. To see that the map $\mathrm{Sym} P_{\sbullet} \to \left(\C[S_{k+l}] \otimes (A[1]^{\otimes k}\otimes M^{\otimes l} \otimes N^{\otimes l})\right)_{S_k\times S_l\times S_l}$ is surjective, we use Proposition \ref{Prop_Sl_coinv} and observe that elements of $S_{k,l}$ are written as product of cycles that contain at most one number in $\{k+1,\dots,k+l\}\subset \{1,\dots,k+l\}$. Thus we have $P_{\sbullet}  =  \mathrm{Prim} \,\mathrm{Sym}(\mathcal{L}[1])_{\mathfrak{gl}} $.

Now we analyze $P_{\sbullet}$. The first summand $(\C[U_{k}] \otimes (A[1]^{\otimes k}))_{S_k} $ in \ref{prim_afund} is isomorphic to $C^{\lambda}_{k-1}(A)$. For the second summand in \ref{prim_afund}, note that $U_{k+1}$ is isomorphic to $S_k$ as an $S_k$ set. Therefore we have
\begin{equation*}
\mathrm{Prim}\,(\mathrm{Sym}(\mathcal{L}[1]) )_{\mathfrak{gl}}   = C^{\lambda}_{\sbullet}(A)[1] \oplus  B_{\sbullet}(A,M\otimes N),
\end{equation*}

In order to finish the proof it remains to identify the Hochschild differential on $B_{\sbullet}(A,M\otimes N)$ with the \CE differential.

By the isomorphism \ref{iso_CE_Hoch}, the image of the element
\begin{equation*}
	s =  E^{a_1}_{12} \otimes E^{a_2}_{23} \dots E^{a_n}_{n\,n+1}\otimes e^{x}_{n+1}\otimes f^{y}_{1} 
\end{equation*}
under $\widetilde{\Theta}$ represent the class of $(a_1\otimes a_2\dots a_n\otimes x \otimes y) \in B_n(A,M\otimes N)$.

The \CE differential is computed as follows:
\begin{equation*}
\begin{aligned}
		d_{CE}(s) =& E^{a_2}_{23}\otimes E^{a_3}_{34} \dots E^{a_n}_{n\,n+1}\otimes e^{x}_{n+1}\otimes f^{ya_1}_{2} \\ 
		&+ \sum_{i=1}^{n-1}(-1)^i E^{a_1}_{12} \dots  \otimes E^{a_ia_{i+1}}_{i,i+2}\otimes\dots E^{a_n}_{n\,n+1}\otimes e^{x}_{n+1}\otimes f^{y}_{1}\\
		& + (-1)^n E^{a_1}_{12} \otimes E^{a_2}_{23} \dots E^{a_{n-1}}_{n-1\,n}\otimes e^{a_n x}_{n}\otimes f^{y}_{1}. \\
\end{aligned}
\end{equation*}
Its image under $\widetilde{\Theta}$ is exactly the Hochschild differential
$$
	b(a_1\otimes \dots a_n\otimes x\otimes y).
$$

Finally, note that
$$
	HH_{\sbullet}(A,M\otimes N) \cong \mathrm{Tor}^A_{\sbullet}(N,M).
$$
\end{proof}

\subsection{Unframed quiver}
Recall that a quiver $Q$ is a directed graph that consist of two sets $Q_0$ (the set of vertices of $Q$), $Q_1$ (the set of edges of $Q$) and two functions
$$
	s,t : Q_1 \rightrightarrows Q_0,
$$
called the source and target functions.

Given a quiver $Q$, we consider the following data of algebras and modules associated with $Q$.
\begin{equation}
	\label{alg_data_quiver}
	 \left\{
	\begin{tabular}{c}
		To each vertex $v \in Q_0$ we assign a unital $k$ algebra $A_v$\\
		To each edge $e \in Q_1$ we assign a $A_{s(e)} - A_{t(e)}$ bimodule $M_e$
	\end{tabular}
	\right\}.
\end{equation}

This allows us to define the direct sum Lie algebra $\bigoplus_{v \in Q_0}\mathfrak{gl}_N(A_v)$ and the $\bigoplus_{v \in Q_0}\mathfrak{gl}_N(A_v)$ module $\bigoplus_{e\in Q_1} \mathrm{Mat}_N(M_e)$. Denote $\mathcal{M} := \lim\limits_{\overset{\longrightarrow}{N}}\mathrm{Mat}_N$. We consider the Lie algebra homology:
\begin{equation}\label{hom_un_Q}
	H_{\sbullet}(\bigoplus_{v \in Q_0}\mathfrak{gl}(A_v),\mathrm{Sym}(\bigoplus_{e\in Q_1}\mathcal{M}(M_e))) = \lim_{\overset{\longrightarrow}{N}} H_{\sbullet}(\bigoplus_{v \in Q_0}\mathfrak{gl}_N(A_v),\mathrm{Sym}(\bigoplus_{e\in Q_1}\mathrm{Mat}_N(M_e))).
\end{equation}
This can be identified with the Lie algebra homology of the following semidirect product Lie algebra
\begin{equation*}
	\mathfrak{g}_{Q,N} := \bigoplus_{v \in Q_0}\mathfrak{gl}_N(A_v)\ltimes \bigoplus_{e\in Q_1}\mathrm{Mat}_N(M_e)[-1].
\end{equation*}
We also define a (super) algebra as an abelian extension of $\bigoplus_{v \in Q_0} A_v$ by $\bigoplus_{e\in Q_1} M_e[-1]$.
\begin{equation*}
	\mathcal{A}_{Q} := \bigoplus_{v \in Q_0} A_v \bigoplus_{e\in Q_1} M_e[-1].
\end{equation*}
By the definition of $\mathcal{A}_Q$, we have the following identification of Lie algebras
\begin{equation*}
	\mathfrak{g}_{Q,N} = \mathfrak{gl}_N(\mathcal{A}_Q) .
\end{equation*}
Then we can further identify the homology \ref{hom_un_Q} with the Lie algebra homology of $\mathfrak{gl}(\mathcal{A}_Q)$. As a corollary of the Loday-Quillen-Tsygan theorem, we have
\begin{cor}
	There is a canonical isomorphism of graded Hopf algebras
	\begin{equation*}
		H_{\sbullet}(\bigoplus_{v \in Q_0}\mathfrak{gl}(A_v),\mathrm{Sym}(\bigoplus_{e\in Q_1}\mathcal{M}(M_e))) \cong \mathrm{Sym}(HC_{\sbullet}(\mathcal{A}_Q )[1]).
	\end{equation*}
\end{cor}

The aim of this section is to compute the cyclic homology $HC_{\sbullet}(\mathcal{A}_Q )$.

Recall that a path of length $n$ in a quiver $Q$ is a sequence $(e_1,e_2,\dots,e_n)$ of edges such that $s(e_{i + 1}) = t(e_i)$ for $i = 1,\dots n - 1$. A pointed cycle is a path that satisfies the additional relation $s(e_1) = t(e_n)$. Denote $\mathrm{Cyc}^\circ_n(Q)$ the set of pointed cycles of length $n$ in $Q$. The cyclic group $\Z_n$ acts on $\mathrm{Cyc}^\circ_n$ via the action $(e_1,e_2,\dots,e_n) \to (e_n,e_1,\dots,e_{n-1})$. We define a cycle as an equivalent class of pointed cycles module the cyclic permutation. We denote $\mathrm{Cyc}_n(Q) : = \mathrm{Cyc}^{\circ}_n(Q)/\Z_n$, and $\mathrm{Cyc}(Q) = \bigcup_{n\geq 1}\mathrm{Cyc}_n(Q)$.

We define the degree $\deg \ell $ of a cycle $\ell$ to be the largest integer $d$ such that there exist a cycle $\ell_0$ and
$$
	\ell =\ell_0^d\,,
$$
where the multiplication is taken in the path algebra of the quiver. For each cycle $\ell$ represented by $(e_1,e_2,\dots,e_n)$ we define the following space
\begin{equation}\label{cyc_tensor}
	\mathcal{F}_{\ell} = H_{\sbullet}\left( M_{e_1}\otimes_{A_{t(e_1)}}^{\mathbb{L}}M_{e_2}\otimes_{A_{t(e_2)}}^{\mathbb{L}}\dots \otimes_{A_{t(e_{n-1})}}^{\mathbb{L}}M_{e_n}\otimes_{A_{t(e_n)}}^{\mathbb{L}}\right) _{\Z_{\deg \ell}}.
\end{equation} 
For any other representative $(e_1',e_2',\dots,e_n')$ of $\ell$, we have a cyclic action $g\in \Z_n$ such that $(e_1',e_2',\dots,e_n') = g(e_1,e_2,\dots,e_n)$. Then we have a canonical isomorphism
$$
	\left( M_{e_1'}\otimes_{A_{t(e_1')}}^{\mathbb{L}}\dots \otimes_{A_{t(e_{n-1}')}}^{\mathbb{L}}M_{e_n'}\otimes_{A_{t(e_n')}}^{\mathbb{L}}\right) _{\Z_{\deg \ell}} \cong \left( M_{e_1}\otimes_{A_{t(e_1)}}^{\mathbb{L}}\dots \otimes_{A_{t(e_{n-1})}}^{\mathbb{L}}M_{e_n}\otimes_{A_{t(e_n)}}^{\mathbb{L}}\right) _{\Z_{\deg \ell}},
$$
induced by the cyclic action $g\in \Z_n$. Therefore, the definition of $\mathcal{F}_{\ell}$ does not depend on the choice of representation of the cycle $\ell$.

In the special case of $\ell = (e)$, we have an algebra $A_{s(e)}$ (note that $s(e) = t(e)$), and a $A_{s(e)}$ bimodule $M_e$. Then the above definition gives us the Hochschild homology with coefficient in the bimodule $M_e$, $\mathcal{F}_{\ell} = HH_{\sbullet}(A_{s(e)}, M_e)$. 

In this section, we prove the following result
\begin{prop}
	\label{Prop_CycH_AQ}
For any quiver $Q$, let $\mathcal{A}_Q $ be defined as above. There is an isomorphism

\begin{equation}
	HC_{\sbullet}(\mathcal{A}_Q ) \cong \bigoplus_{v \in Q_0} HC_{\sbullet}(A_v) \bigoplus_{\ell \in \mathrm{Cyc}(Q)} \mathcal{F}_{\ell}[-1].
\end{equation}
\end{prop}
\begin{proof}
	Denote $\mathbb{1}_v : = \mathbb{1}_{A_v}$ the unit of the algebra $A_v$. We have $\mathbb{1}_v\cdot \mathbb{1}_{v'} = 0$ if $v \neq v'$ and $\mathbb{1}_v^2 = \mathbb{1}_v $.
	Let $S = \bigoplus_{v \in Q_0}k\mathbb{1}_{v} \subset \mathcal{A}_Q$. $S$ is a separable algebra with idempotent $e = \sum_{v \in Q_0}\mathbb{1}_{v} \otimes \mathbb{1}_{v}$. To compute the cyclic homology, we use Corollary \ref{Cor_cyc} and consider the Connes' complex $C^{\lambda,S}_{\sbullet}(\mathcal{A}_Q)$ over $S$. We have 
$$
		HC_{n}(\mathcal{A}_Q) \cong H_{n}(C_{\sbullet}^{\lambda,S}(\mathcal{A}_Q) ).
$$
Let $Q^+$ be the quiver constructed from $Q$ by adding to each vertex $v \in Q_0$ an edge $v^*$ that has $v$ as both its source and target.  In other words,
\begin{equation}
	\label{Q+}
	Q^{+} = (Q_0,Q_1\sqcup Q_0).
\end{equation}
For each vertex $v \in Q_0$, denote $v^*$ the corresponding edge in $Q^+$. We have $t(v^*) = s(v^*) = v$.

We observe that the tensor product over $S$ gives either $0$ or the tensor product over $k$. Denote $M_{v*} : = A_v[1]$. We have
\begin{equation}
	\label{tensor_S}
M_e\otimes_S M_{e'}  = \begin{cases}
	0 &\text{ if } t(e) \neq s(e') \text{ in } Q^{+},\\
	M_e\otimes M_{e'} & \text{ if } t(e) = s(e') \text{ in } Q^{+}.
\end{cases} 
\end{equation}
Therefore, we can rewrite the $n$-fold tensor product $$(\mathcal{A}_Q[1])^{\otimes_S n}\otimes_S : = \underbrace{\mathcal{A}_Q[1]\otimes_S\mathcal{A}_Q[1]\otimes_S\dots \mathcal{A}_Q[1]\otimes_{S}}_{n}$$ as follows
$$
	(\mathcal{A}_Q[1])^{\otimes_S n}\otimes_S  = \bigoplus_{\tilde{\ell} = (\tilde{e}_1,\dots,\tilde{e}_n) \in \mathrm{Cyc}^\circ_n(Q^+)} M_{\tilde{e}_1}\otimes M_{\tilde{e}_2} \otimes \dots M_{\tilde{e}_n}.
$$
 Under the above decomposition, the group $\mathbb{Z}_n$ acts on the summand by isomorphism
\begin{equation}
	\label{iso_cyc}
		t: M_{\tilde{e}_1}\otimes M_{\tilde{e}_2} \otimes \dots M_{\tilde{e}_n} \overset{\simeq}{\to} M_{\tilde{e}_n}\otimes M_{\tilde{e}_1} \otimes \dots M_{\tilde{e}_{n-1}},
\end{equation}
which is compatible with the cyclic action on the pointed cycle, $t:  (\tilde{e}_1,\tilde{e}_2,\dots,\tilde{e}_n)  \mapsto  (\tilde{e}_n,\tilde{e}_1,\dots,\tilde{e}_{n-1}) $. Then by passing to the coinvariant space, we can express the above direct sum as summing over cycles in $\mathrm{Cyc}(Q^+)$.

Note that when $\deg \tilde{\ell}  = d \geq 2$, $ \tilde{\ell} = \ell_0^d$ for some $\ell_0$. In this case a representative of $ \tilde{\ell}$ is fixed by a nontrivial subgroup $\Z_d \subset \Z_n$, generated by $t^{n/d}:  \tilde{\ell} \mapsto  \tilde{\ell}$. Therefore, we find that the Connes' complex is isomorphic to the following
\begin{equation}
	\label{Cyc_cpx+}
	C^{\lambda,S}_{\sbullet-1}(\mathcal{A}_Q) = \bigoplus_{\tilde{\ell} = (\tilde{e}_1,\dots,\tilde{e}_n) \in \mathrm{Cyc}(Q^+)} (M_{\tilde{e}_1}\otimes M_{\tilde{e}_2} \otimes \dots M_{\tilde{e}_n})_{\Z_{\deg  \tilde{\ell}}}.
\end{equation}
where we can choose any representative $(\tilde{e}_1,\dots,\tilde{e}_n)$ of the class $ \tilde{\ell}\in \mathrm{Cyc}(Q^+)$ using the isomorphism \ref{iso_cyc}.

By the construction of $Q^+$, any $\tilde{\ell}  \in \mathrm{Cyc}(Q^+)$ can be represented either by 
\begin{equation*}
	(v^*)^i\quad \text{ for some } v\in Q_0,
\end{equation*}
or by 
\begin{equation}\label{Q+_to_Q}
		(s(e_1)^*)^{i_1}e_1(s(e_2)^*)^{i_2}e_2\dots e_m \,\;\;\ \text{for some } \ell =(e_1,e_2,\dots, e_m) \in \mathrm{Cyc}_m(Q) .
\end{equation}
In the first case, we obtain the complex $C^{\lambda}_{\sbullet-1}(A_v)$ that computes the cyclic homology of $A_v$. We focus on the second case in the following.

We would like to use \ref{Q+_to_Q} to rewrite the summation \ref{Cyc_cpx+} as a summation over $\ell \in \mathrm{Cyc}_m(Q)$ and $i_1,\dots,i_m \in \Z_+$. However, the map
\begin{equation*}
 \varphi: \bigcup_{m \geq 1}\mathrm{Cyc}_m(Q)\times \Z_+^m \to \mathrm{Cyc}(Q^+)
\end{equation*}
defined by \ref{Q+_to_Q} is not injective. To correct this, we use a larger group for the coinvariant. We take any $\tilde{\ell} \in \mathrm{Cyc}(Q^+)$ and the corresponding $\ell \in \mathrm{Cyc}(Q)$ such that $\varphi(\ell,i_1,\dots,i_m) = \tilde{\ell}$. We observe that the group $\Z_{\deg \ell}$ acts transitively on the set of preimages $\varphi^{-1}(\tilde{\ell})$. Moreover, the subgroup that fix $(\ell,i_1,\dots,i_m)$ is exact $\Z_{\deg  \tilde{\ell}}$. Therefore,  for $\ell \in \mathrm{Cyc}(Q)$, we define the following $\text{length}(\ell)$-fold complex
\begin{equation*}
	C(Q,\ell)_{\sbullet,\dots,\sbullet} = \left( \bigoplus_{i_1,\dots,i_{m} \geq 0} (A_{s(e_1)}[1])^{\otimes i_1}\otimes M_{e_1}\otimes( A_{s(e_2)}[1])^{\otimes i_2}\otimes \dots M_{e_m}\right)_{\Z_{\deg \ell}}.
\end{equation*}
Combining with the first case when $\tilde{\ell} = (v^*)^i$, we have the following isomorphism
\begin{equation*}
	\bigoplus_{\tilde{\ell} = (\tilde{e}_1,\dots,\tilde{e}_n) \in \mathrm{Cyc}(Q^+)} (M_{\tilde{e}_1}\otimes M_{\tilde{e}_2} \otimes \dots M_{\tilde{e}_n})_{\Z_{\deg  \tilde{\ell}}}\cong  \bigoplus_{v \in Q_0} C^{\lambda}_{\sbullet}(\mathcal{A}_v) [1]\bigoplus_{m\geq 1}\bigoplus_{\ell \in \mathrm{Cyc}_m(Q)} \mathrm{Tot} \,C(Q,\ell).
\end{equation*}
Then we find  that the Connes' cyclic complex is isomorphic to
\begin{equation*}
 	C^{\lambda,S}_{\sbullet}(\mathcal{A}_Q) =  \bigoplus_{v \in Q_0} C^{\lambda}_{\sbullet}(\mathcal{A}_v) \bigoplus_{\ell \in \mathrm{Cyc}(Q)} \mathrm{Tot} \,C(Q,\ell)[-1].
\end{equation*}

Finally, we analyze the differential of the complex $C(Q,\ell)$. The differential is induced from the differential $b$ on the Connes' complex $C_{\sbullet}^{\lambda,S}(\mathcal{A}_Q)$.  For any $x_{j}\in M_{e_{j}}$ and $a^j_1,\dots, a^j_{i_j} \in A_{s(e_j)},a^{j+1}_1,\dots, a^{j+1}_{i_{j+1}} \in A_{s(e_{j+1})}$ , we have
\begin{equation}
\begin{aligned}
		&b(\dots \otimes a^j_1\otimes a^j_2\otimes \dots\otimes a^j_{i_j} \otimes x_j\otimes a^{j+1}_1\otimes \dots \otimes a^{j+1}_{i_{j+1}}\otimes \dots )\\
		 & = \dots \dots \otimes a^j_1a^j_2\otimes \dots\otimes a^j_{i_j} \otimes x_j\otimes a^{j+1}_1\otimes \dots \otimes a^{j+1}_{i_{j+1}}\otimes \dots \\
		 &+\dots\\
		&  + \dots \otimes a^j_1\otimes a^j_2\otimes \dots\otimes a^j_{i_j}  x_j\otimes a^{j+1}_1\otimes \dots \otimes a^{j+1}_{i_{j+1}}\otimes \dots\\
		& + \dots \otimes a^j_1\otimes a^j_2\otimes \dots\otimes a^j_{i_j} \otimes x_j a^{j+1}_1\otimes \dots \otimes a^{j+1}_{i_{j+1}}\otimes \dots\\
		& + \dots
\end{aligned}
\end{equation}
We find that the differential on $C(Q,\ell)$ correspond to the differential of the Bar complex computing the derived tensor product $ \left( M_{e_1}\otimes_{A_{t(e_1)}}^{\mathbb{L}}M_{e_2}\otimes_{A_{t(e_2)}}^{\mathbb{L}}\dots \otimes_{A_{t(e_{n-1})}}^{\mathbb{L}}M_{e_n}\otimes_{A_{t(e_n)}}^{\mathbb{L}}\right) _{\Z_{\deg \ell}}$.

\end{proof}

\begin{eg}
	In the simplest case, we consider the quiver with one vertex $v$ and one edge $e$ starting and ending on $v$. For this quiver, we assign an associative algebra $A$ to the vertex and a $A$-$A$ bimodule $M$ to the edge.
			
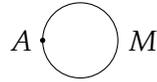
\begin{center}
		\begin{tikzpicture}[scale = 0.5]
			\draw(0,0) circle (1);
			\fill (-1,0) circle (2pt) node[left] {$A$};
			\draw (1,0) node[right] {$M$};
		\end{tikzpicture}
	\end{center}

In this case, we find that the Lie algebra homology $H_{\sbullet}(\mathfrak{gl}(A),\mathrm{Sym}(\mathcal{M}(M)))$ is given by the following
\begin{equation*}
	\mathrm{Sym}(HC_{\sbullet}(A)[1]\oplus H_{\sbullet}(A,M) \oplus H_{\sbullet}(M\otimes^{\mathbb{L}}_AM\otimes^{\mathbb{L}}_A)_{\Z_2}\oplus \cdots).
\end{equation*}
\end{eg}

\subsection{Framed quiver}
A framed quiver $Q^{\text{fr}}$ consists of the data of a quiver $Q = (Q_0,Q_1)$ together with two subsets $W^+,W^-$ of $Q_0$. Denote the inclusion map $j:W^+\sqcup W^- \to Q_0$. We can think of the framed quiver $Q^{\text{fr}}$ as adding to $Q$ extra vertices $W^+,W^-$ and edges $w^+ \to j(w^+)$ for all $w^+ \in W^+$, and $ j(w^-) \to w^-$ for all $w^- \in W^-$.

For a framed quiver $Q^{\text{fr}} = (Q,W^+,W^-)$, we consider the same set of data \ref{alg_data_quiver} for $Q$, together with the following
\begin{equation}
	\left\{
	\begin{tabular}{c}
		To each vertex $w^+ \in W^+$ we assign a right $A_{j(w^+)}$ module $M_{w^+}$\\
		To each vertex $w^- \in W^-$ we assign a left $A_{j(w^-)}$ module $M_{w^-}$
	\end{tabular}
	\right\}.
\end{equation}

In the last section, we defined the algebra $\mathcal{A}_Q$ associated with the quiver $Q$. Given the above data associated with the framing, we denote $$M_+ = \bigoplus_{w^+ \in W^+}M_{w^+},\quad M_- = \bigoplus_{w^- \in W^-}M_{w^-}.$$ Then $M_+$ is a right $\mathcal{A}_Q$ module and $M_-$ is a left $\mathcal{A}_Q$ module. Using the construction of \ref{fund_mod} and \ref{afund_mod}, we have a $\mathfrak{gl}_N(\mathcal{A}_Q)$ module $\C^{N}(M_-)\bigoplus (\C^N)^*(M_+).$
In this section, we compute the following Lie algebra homology
\begin{equation*}
	H_{\sbullet}(\mathfrak{gl}(\mathcal{A}_Q),\mathrm{Sym}(\mathcal{V}(M_-)\oplus \mathcal{V}^*(M_+))) = \lim_{\overset{\longrightarrow}{N}} H_{\sbullet}(\mathfrak{gl}_N(\mathcal{A}_Q),\mathrm{Sym}(\C^N(M_-)\oplus (\C^N)^*(M_+))).
\end{equation*}
As a corollary of Theorem \ref{LQT_addf}, we have the following
\begin{cor}
	We have a canonical isomorphism of graded Hopf algebras
	\begin{equation*}
	H_{\sbullet}(\mathfrak{gl}(\mathcal{A}_Q),\mathrm{Sym}(\mathcal{V}(M_-)\oplus \mathcal{V}^*(M_+))) \cong \mathrm{Sym}(HC_{\sbullet}(\mathcal{A}_Q) [1]\oplus H_{\sbullet}(M_+\otimes_{\mathcal{A}_Q}^{\mathbb{L}}M_-)).
	\end{equation*}
\end{cor}

The cyclic homology $HC_{\sbullet}(\mathcal{A}_Q) $ is computed in the last section. In this section, we compute $M_+\otimes_{\mathcal{A}_Q}^{\mathbb{L}}M_-$.

Denote $\mathrm{fPath}_{w^+,w^-}(Q^{\mathrm{fr}})$ the set of paths in $Q^{\text{fr}}$ that start with $w^+$ and end with $w^-$. Let $\mathrm{fPath}(Q^{\mathrm{fr}})= \cup_{w^+ \in W^+,w^- \in W^-}\mathrm{fPath}_{w^+,w^-}$. To each $\rho = (w^+,e_1,\dots,e_n,w^-) \in \mathrm{fPath}(Q^{\mathrm{fr}})$, we assign the space
\begin{equation}
	\mathcal{F}_\rho = H_{\sbullet}\left(M_{w^+}\otimes^{\mathbb{L}}_{A_{j(w^+)}}M_{e_{1}}\otimes_{A_{t(e_1)}}^{\mathbb{L}} \dots M_{e_n}\otimes_{A_{j(w^-)}}^{\mathbb{L}} M_{w^-}\right).
\end{equation}

Then we have the following result
\begin{prop}
	\begin{equation}
		 H_{\sbullet}(M_+\otimes^{\mathbb{L}}_{\mathcal{A}_Q}\otimes M_-) = \bigoplus_{\rho \in \mathrm{fPath}} \mathcal{F}_\rho .
	\end{equation}
\end{prop}
\begin{proof}
	To compute the derived tensor product $M_+\otimes^{\mathbb{L}}_{\mathcal{A}_Q}\otimes M_- $, we use Corollary \ref{Cor_Tor_S} and consider the Bar complex over $S$
	$$
		B^S_{\sbullet}(M_+,\mathcal{A}_Q,M_-) = \bigoplus_{n\geq 0}M_+\otimes_S(\mathcal{A}_Q[1])^{\otimes_S n} \otimes_S M_-.
	$$
	Let $Q^{\mathrm{fr},+} = (Q^+,W^+,W^-)$, where $Q^+$ is constructed from $Q$ by \ref{Q+}. Namely, $Q^{\mathrm{fr},+} $ is constructed from the quiver $Q^{\mathrm{fr}}$ by adding to each non-framing vertex $v\in Q_0$ an edge $v^*$ that has $v$ as both its source and target. The tensor product $\otimes_S$ has the same property as \ref{tensor_S}. Therefore, we have the following isomorphism
	\begin{equation}
		B^S_{\sbullet}(M_+,\mathcal{A}_Q,M_-)  = \bigoplus_{\rho^+ = (w^+,e_1,\dots,e_n,w^-)  \in \mathrm{fPath}(Q^{\mathrm{fr},+})} M_{w^+}\otimes M_{e_1} \otimes \dots M_{w^-},
	\end{equation}
	where we denote $M_{v^*} = A_{v}[1]$ as before. 
	
	By definition, each framed path $\rho^+ \in \mathrm{fPath}(Q^{\mathrm{fr},+}) $ can be expressed as follows
$$
		w^+(s(e_1)^*)^{i_1}e_1(s(e_2)^*)^{i_2}e_2\dots e_m(t(e_m)^*)^{i_{m+1}} w^- \,\;\;\ \text{for some } \ell = (e_1,e_2,\dots, e_m) \in \mathrm{Cyc}_m(Q) .
$$
As a result, the Bar complex can be written as follows
$$
B^S_{\sbullet}(M_+,\mathcal{A}_Q,M_-)  = \bigoplus_{\rho  \in \mathrm{fPath}(Q^{\mathrm{fr}})} C(Q^{\mathrm{fr}},\rho),
$$
where $C(Q^{\mathrm{fr}},\rho)$ is the complex 
\begin{equation*}
	C(Q^{\mathrm{fr}},\rho)_{\sbullet,\dots,\sbullet} =  \bigoplus_{i_1,\dots,i_{m} \geq 0} M_{w^+}\otimes (A_{s(e_1)}[1])^{\otimes i_1}\otimes M_{e_1}\otimes  \dots \otimes (A_{t(e_m)}[1])^{\otimes i_{m+1}}\otimes M_{w^-}.
\end{equation*}
We find that the complex $C(Q^{\mathrm{fr}},\rho)_{\sbullet,\dots,\sbullet}$ is exactly the Bar complex computing the derived tensor product $M_{w^+}\otimes^{\mathbb{L}}_{A_{j(w^+)}}M_{e_{1}}\otimes_{A_{t(e_1)}}^{\mathbb{L}} \dots M_{e_n}\otimes_{A_{j(w^-)}}^{\mathbb{L}} M_{w^-}$. Therefore, we have the following 
	\begin{equation}
		\mathrm{Tor}^{\mathcal{A}_Q}_{\sbullet}(M_+,M_-) = \bigoplus_{\rho \in \mathrm{fPath}(Q^{\mathrm{fr},+})} \mathcal{F}_\rho .
	\end{equation}
\end{proof}
Together with the Proposition \ref{Prop_CycH_AQ} that compute the cyclic homology of $\mathcal{A}_Q$, we conclude that
\begin{theorem}\label{thm_main1}
	We have an isomorphism of graded Hopf algebras
$$
		\begin{aligned}
			H_{\sbullet}(\bigoplus_{v \in Q_0}\mathfrak{gl}(A_v),&\mathrm{Sym}(\bigoplus_{e\in E}\mathcal{M}(M_e)  \bigoplus_{w^-\in W^-} \mathcal{V}(M_{w^-})\bigoplus_{w^+ \in W^+} \mathcal{V}^*(M_{w^+})) )\\
			&\cong \mathrm{Sym} \left(  \bigoplus_{v \in Q_0}HC_{\sbullet}(A_v)[1] \bigoplus_{\ell \in \mathrm{Cyc}(Q)}\mathcal{F}_{\ell} \bigoplus_{\rho \in \mathrm{fPath}}\mathcal{F}_{\rho}\right) .
		\end{aligned}
$$
\end{theorem}

\subsection{Factorization homology}
In this section, we show that our previous result can be succinctly reformulated in terms of (stratified) factorization homology.

As a first step, we see that the data of \ref{alg_data_quiver} assigns a vertex $v \in Q_0$ an associative algebra $A_v$, which can be further identified with a one dimensional factorization algebra. We have the isomorphism between the Hochschild homology of $A_v$ and the factorization homology of $A_v$ over $S^1$\cite{Lurie}:
\begin{equation}
	\int_{S^1}A_v \cong HH_{\sbullet}(A).
\end{equation}

The natural $S^1$ action that rotates the circle induces an action of $S^1$ on the factorization homology $\int_{S^1}A_v$. Moreover, this action can be identified with the Connes’ cyclic operator on the Hochschild bar complex \cite{Ayala2021Symmetry}. Taking homotopy orbits (coinvariant) of the $S^1$-action on the Hochschild homology gives us the cyclic homology. We have
\begin{equation*}
\left( \int_{S^1}A_v\right) _{S^1} \cong HC_{\sbullet}(A_v).
\end{equation*}

Next, we consider a cycle $\ell  = (e_1,\dots,e_m) \in \mathrm{Cyc}(Q)$ of the quiver. This gives us a set of associative algebras $(A_{t(e_i)})_{i=1,\dots,n}$ together with the set of $A_{s(e_i)}-A_{t(e_i)}$ bimodules $M_{e_i}$. In fact, this data $(A_{t(e_i)},M_{e_i})$ defines a stratified factorization algebra $\mathcal{A}_\ell$ on $S^1$ with $m$ marked points. 

		
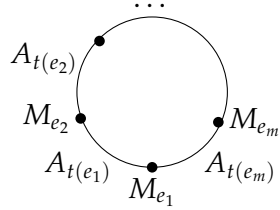
\begin{center}
	\begin{tikzpicture}
		\draw(0,0) circle (1);
		\fill (0,-1) circle (2pt) node[below] {$M_{e_1}$};
		\draw (-0.4,-1) node[left] {$A_{t(e_1)}$};
		\fill (-0.95,-0.35) circle (2pt) node[left] {$M_{e_2}$};
		\draw (-0.85,0.45) node[left] {$A_{t(e_2)}$};
		\fill (-0.7,0.68) circle (2pt);
		\draw (0,1)  node[above] {$\dots $};
		\fill (0.88,-0.4) circle (2pt) node[right] {$M_{e_m}$};
		\draw (0.55,-1) node[right]{$A_{t(e_m)}$};
	\end{tikzpicture}
\end{center}

Using the excision theorem for stratified factorization algebra \cite{Ayala2014FactorizationHO}, we have the following result.
\begin{prop}
	There is an equivalence
	\begin{equation*}
		\int_{S^1} \mathcal{A}_\ell =  M_{e_1}\otimes_{A_{t(e_1)}}^{\mathbb{L}}M_{e_2}\otimes_{A_{t(e_2)}}^{\mathbb{L}}\dots \otimes_{A_{t(e_{n-1})}}^{\mathbb{L}}M_{e_n}\otimes_{A_{t(e_n)}}^{\mathbb{L}}.
	\end{equation*}
\end{prop}

In the presence of the marked points, only the discrete subgroup $\Z_{\deg \ell} \subset S^1$ acts on the factorization homology $\int_{S^1} \mathcal{A}_\ell$. Taking the coinvariant of this action gives us \ref{cyc_tensor}.

Finally, we consider a path $\rho = (w^+,e_1,\dots,e_n,w^-) \in \mathrm{fPath}$ that ends on framed vertices. This gives us a set of associative algebras $(A_{s(e_1)},A_{s(e_2)},\dots,A_{s(e_n)},A_{t(e_n)})$, a set of $A_{s(e_i)}-A_{t(e_i)}$ bimodules $M_{e_i}$, together with a right $A_{s(e_1)}$ module $M_{w^+}$ and a left $A_{t(e_n)}$ module $M_{w^-}$. This data defines a stratified factorization algebra over the interval $[0,1]$ with $n$ additional marked points.

		
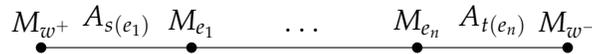
\begin{center}
	\begin{tikzpicture}
		\draw (0,0) -- (7,0);
		\fill (0,0) circle (2pt) node[above] {$M_{w^+}$};
		\draw (1,0) node[above] {$A_{s(e_1)}$};
		\fill (2,0) circle (2pt) node[above] {$M_{e_1}$};
		\draw (3.5,0) node[above] {$\cdots$};
		\fill (5,0) circle (2pt) node[above] {$M_{e_n}$};
		\draw (6,0) node[above] {$A_{t(e_n)}$};
		\fill (7,0) circle (2pt) node[above] {$M_{w^-}$};
	\end{tikzpicture}
\end{center}

Again, we use the excision theorem for stratified factorization algebra \cite{Ayala2014FactorizationHO} and find the following result.
\begin{prop}
	There is an equivalence
	\begin{equation*}
		\int_{[0,1]} \mathcal{A}_\rho =   M_{w^+}\otimes^{\mathbb{L}}_{A_{j(w^+)}}M_{e_{1}}\otimes_{A_{t(e_1)}}^{\mathbb{L}} \dots M_{e_n}\otimes_{A_{j(w^-)}}^{\mathbb{L}} M_{w^-} .
	\end{equation*}
\end{prop}

Combining our previous analysis, we can now reformulate the main Theorem \ref{thm_main1} of this paper using factorization homology
\begin{theorem}\label{thm_main2}
	We have an isomorphism 
	$$
	\begin{aligned}
		&H_{\sbullet}(\bigoplus_{v \in Q_0}\mathfrak{gl}(A_v),\mathrm{Sym}(\bigoplus_{e\in E}\mathcal{M}(M_e)  \bigoplus_{w^-\in W^-} \mathcal{V}(M_{w^-})\bigoplus_{w^+ \in W^+} \mathcal{V}^*(M_{w^+})) )\\
		&\cong \mathrm{Sym} \left(  \bigoplus_{v \in Q_0}\left( \int_{S^1}A_v\right)_{S^1}[1] \oplus  \bigoplus_{\ell \in \mathrm{Cyc}(Q)} \left( \int_{S^1}\mathcal{A}_{\ell}\right) _{\Z_{\deg \ell}}\oplus \bigoplus_{\rho  \in \mathrm{fPath}(Q)} \left( \int_{[0,1]}\mathcal{A}_{\rho}\right) \right).
	\end{aligned}
	$$
\end{theorem}

\end{document}